\documentclass[a4paper,12pt,final]{amsart}
\usepackage{times,a4wide,mathrsfs,amssymb,amsmath,amsthm}

\newcommand{\C}{\mathbb{C}}
\newcommand{\ZZ}{\mathbb{Z}}

\newcommand{\QQ}{\mathbb{Q}}

\newcommand{\PP}{\mathbb{P}}

\newcommand{\Sy}{\mathfrak S}

\newcommand{\MM}{\mathcal M}

\DeclareMathOperator{\aut}{Aut}

\newtheorem{theorem}{Theorem}[section]
\newtheorem{claim}[theorem]{Claim}

\newtheorem{corollary}[theorem]{Corollary}
\newtheorem{proposition}[theorem]{Proposition}
\newtheorem{conjecture}[theorem]{Conjecture}

\newtheorem{nonumbering}{Theorem}

\newtheorem{nonumberingc}{Corollary}

\theoremstyle{definition}
\newtheorem{remark}[theorem]{Remark}
\newtheorem{definition}[theorem]{Definition}
\newtheorem{convention}{Conventions}

\newtheorem{nonumberingt}{Acknowledgements}

\begin{document}
\author[Robert Laterveer]
{Robert Laterveer}

\address{Institut de Recherche Math\'ematique Avanc\'ee,
CNRS -- Universit\'e 
de Strasbourg,\
7 Rue Ren\'e Des\-car\-tes, 67084 Strasbourg CEDEX,
FRANCE.}
\email{robert.laterveer@math.unistra.fr}

\title{Zero-cycles on Cancian--Frapporti surfaces}

\begin{abstract} An old conjecture of Voisin describes how $0$-cycles on a surface $S$ should behave when pulled-back to the self-product 
$S^m$ for $m>p_g(S)$.
We show that Voisin's conjecture is true for a $3$-dimensional family of surfaces of general type with $p_g=q=2$ and $K^2=7$ constructed by
Cancian and Frapporti, and revisited by Pignatelli--Polizzi.
\end{abstract}


\maketitle

\section{Introduction}

Let $X$ be a smooth projective variety over $\C$, and let $A^i(X)_{\ZZ}:=CH^i(X)_{}$ denote the Chow groups of $X$ (i.e. the groups of codimension $i$ algebraic cycles on $X$ with $\ZZ$-coefficients, modulo rational equivalence \cite{F}). Let $A^i_{hom}(X)_{\ZZ}$ (and $A^i_{AJ}(X)_{\ZZ}$) denote the subgroup of homologically trivial (resp. Abel--Jacobi trivial) cycles.     

The Bloch--Beilinson--Murre conjectures describe an alluring kind of paradise, in which Chow groups are precisely determined by cohomology and the coniveau filtration \cite{J2}, \cite{J4}, \cite{Mur}, \cite{Kim}, \cite{MNP}, \cite{Vo}. The following particular glimpse of this paradise was first formulated by Voisin:

\begin{conjecture}[Voisin 1993 \cite{V9}]\label{conj} Let $S$ be a smooth projective surface. Let $m$ be an integer strictly larger than the geometric genus $p_g(S)$. Then for any $0$-cycles $a_1,\ldots,a_m\in A^2_{AJ}(S)_{\ZZ}$, one has
  \[ \sum_{\sigma\in\Sy_m} \hbox{sgn}(\sigma) a_{\sigma(1)}\times\cdots\times a_{\sigma(m)}=0\ \ \ \hbox{in}\ A^{2m}(S^m)_{\ZZ}\ .\]
  (Here $\Sy_m$ is the symmetric group on $m$ elements, and $ \hbox{sgn}(\sigma)$ is the sign of the permutation $\sigma$.
  The notation $a_1\times\cdots\times a_m$ is shorthand for the $0$-cycle $(p_1)^\ast(a_1)\cdot (p_2)^\ast(a_2)\cdots (p_m)^\ast(a_m)$ on 
  $S^m$, where the $p_j\colon S^m\to S$ are the various projections.)
  \end{conjecture}
  
For surfaces of geometric genus $0$, conjecture \ref{conj} reduces to Bloch's conjecture \cite{B}. As for geometric genus $1$, Voisin's conjecture is still open for a general K3 surface;
examples of surfaces of geometric genus $1$ verifying the conjecture are given in \cite{V9}, \cite{16.5}, \cite{19}, \cite{21}. Examples
of surfaces with geometric genus strictly larger than $1$ verifying the conjecture are given in \cite{32}.
One can also formulate versions of conjecture \ref{conj} for higher-dimensional varieties; this is studied in \cite{V9}, \cite{17}, \cite{24.4}, \cite{24.5}, \cite{BLP}, \cite{LV}, \cite{Ch}, \cite{Bur}.

The modest goal of this note is to add to the stock of surfaces verifying conjecture \ref{conj}, by considering {\em Cancian--Frapporti surfaces\/}.
These are minimal surfaces $S$ of general type with $p_g(S)=q(S)=2$ and $K_S^2=7$ constructed as {\em semi-isogenous mixed surfaces\/} in \cite{CF} and revisited in \cite{PP}.\footnote{As explained in loc. cit., only two families of minimal surfaces of general type with invariants $p_g=q=2$ and $K^2=7$ are known: the $3$-dimensional family of Cancian--Frapporti, and a $2$-dimensional family (distinct from the first family) constructed as bidouble covers by Rito \cite{Rit}. For Rito's surfaces, proving conjecture \ref{conj} seems difficult as they are not known to have finite-dimensional motive.} The main result of this note is:

\begin{nonumbering}[=theorem \ref{main}] Let $S$ be a Cancian--Frapporti surface.
Then conjecture \ref{conj} is true for $S$.
\end{nonumbering}

This is proven by exploiting the facts that Cancian--Frapporti surfaces have (a) finite-dimensional motive (in the sense of \cite{Kim}) and (b) surjective Albanese morphism \cite{PP}. 
A key ingredient of the argument is a strong form of the generalized Hodge conjecture for self-products of abelian surfaces \cite{Ab}, \cite{Ch}.
Because of the use of this key ingredient, I am not sure whether the argument can be adapted to other surfaces with $p_g=q=2$ verifying (a) and (b) (cf. remark \ref{notsure}).

 As a corollary, certain instances of the generalized Hodge conjecture are verified:
 
 \begin{nonumberingc}[=corollary \ref{ghc}] Let $S$ be a Cancian--Frapporti surface, and let $m>2$. Then the sub-Hodge structure
   \[ \wedge^m H^2(S,\QQ)\ \subset\ H^{2m}(S^m,\QQ) \]
   is supported on a divisor.
   \end{nonumberingc}

   \vskip0.6cm

\begin{convention} In this note, the word {\sl variety\/} will refer to a reduced irreducible scheme of finite type over $\C$. A {\sl subvariety\/} is a (possibly reducible) reduced subscheme which is equidimensional. 

{\bf Unless indicated otherwise, all Chow groups will be with rational coefficients}: we will denote by $A_j(X)$ the Chow group of $j$-dimensional cycles on $X$ with $\QQ$-coefficients (and by $A_j(X)_{\ZZ}$ the Chow groups with $\ZZ$-coefficients); for $X$ smooth of dimension $n$ the notations $A_j(X)$ and $A^{n-j}(X)$ are used interchangeably. 

The notations $A^j_{hom}(X)$, $A^j_{AJ}(X)$ will be used to indicate the subgroups of homologically trivial, resp. Abel--Jacobi trivial cycles.
The contravariant category of Chow motives (i.e., pure motives with respect to rational equivalence as in \cite{Sc}, \cite{MNP}) will be denoted $\MM_{\rm rat}$.

\end{convention}

 \section{Cancian--Frapporti surfaces}
 
 \begin{theorem}[Cancian--Frapporti \cite{CF}, Pignatelli--Polizzi \cite{PP}]\label{pp} There exist minimal surfaces $S$ of general type with $p_g(S)=q(S)=2$ and $K_S^2=7$, and surjective Albanese map (of degree $3$).
 These surfaces fill out a dense open subset of a $3$-dimensional component of the Gieseker moduli space of general type minimal surfaces with these invariants.
  \end{theorem}
  
  \begin{proof} We present a condensed outline of the construction, following \cite{PP}.
  
  Let $C_4\subset\PP^3$ be a genus $4$ curve defined as a smooth complete intersection
  \[ r(x_0,x_1)+x_2 x_3= s(x_0,x_1)+x_2^3 + x_3^3=0\ ,\]
  where $r(x_0,x_1), s  (x_0,x_1)$ are homogeneous polynomials of degree $2$ resp. $3$. The curve $C_4$ admits a free action of an order $3$ automorphism $\xi$ defined as
   \[ \xi [ x_0, x_1,x_2,x_3]\:= [ x_0, x_1,\nu\, x_2,\nu^2\, x_3]  \ ,\]
   where $\nu$ is a primitive third root of unity. The quotient $C_2:= C_4/\langle\xi\rangle$ is a smooth genus $2$ curve.
   
   The product $C_4\times C_4$ admits an involution $\sigma$ (switching the two factors) and an order $3$ diagonal automorphism $\xi_{xy}$
   (acting as $\xi$ on both factors). The surface $S$ is now defined as a quotient
   \[ S:= (C_4\times C_4)/G := (C_4\times C_4)/\langle \xi_{xy},\sigma\rangle\ .\]
   (The surface $S$ is smooth, because it is a {\em semi-isogenous mixed surface\/} in the sense of \cite[Definition 2.1]{CF}, cf. \cite[Corollary 1.11]{CF}.)
   
   The group $G$ is a non-normal, abelian subgroup of the group
     \[ H:=\langle \xi_x, \xi_y,\sigma\rangle\ \ \subset\ \aut(C_4\times C_4)\ ,\]
     where $\xi_x,\xi_y$ act as $\xi$ on the first, resp. second, factor.
     As shown in \cite[(4)]{PP}, there is a commutative diagram
     \[ \begin{array}[c]{ccc}
          C_4\times C_4 &&\\
          &&\\
          \downarrow&&\\
          &&\\
          (C_4\times C_4)/\langle \xi_{xy}\rangle &\to&  (C_4\times C_4)/\langle \xi_{x},\xi_y\rangle    \ \cong    C_2\times C_2\\
          &&\\
          \downarrow&&\downarrow\\
          &&\\
          S:=  (C_4\times C_4)/\langle \xi_{xy},\sigma\rangle&\xrightarrow{\beta}&   Y:=(C_4\times C_4)/H\cong \hbox{Sym}^2 (C_2)\\
          &&\\
           &  \searrow{\scriptstyle \alpha} &\ \ \ \downarrow{\scriptstyle \pi}\\ 
           &&\\
           & &  A=\hbox{Alb}(S)\cong\hbox{Jac}(C_2)\ .\\
           \end{array}\]
     Here, the unnamed horizontal arrows are the natural quotient morphisms, the morphism $\pi$ is the contraction of the unique rational curve contained in $Y$, and the morphism $\alpha$ is the Albanese map. The fact that the morphism $\alpha$ making the diagram commute is the Albanese map (which is thus surjective) is contained in \cite[Proposition 1.8]{PP}.

   The invariants of $S$ and the minimality are justified in \cite[Proposition 1.5]{PP}.    
   Finally, the statement about the moduli space is \cite[Theorem 2.7]{PP}.
     \end{proof}
  
  \begin{definition} We will call surfaces as in theorem \ref{pp} {\em Cancian--Frapporti surfaces\/}.
   \end{definition}

 \section{Transcendental part of the motive of a surface}
 
 \begin{theorem}[Kahn--Murre--Pedrini \cite{KMP}]\label{tr} Let $S$ be a smooth projective surface. There exists a decomposition
  \[ h(S)=  h^0(S)\oplus h^1(S)\oplus h^2_{tr}(S)\oplus h^2_{alg}(S)\oplus h^3(S)\oplus h^4(S)   \ \in \MM_{\rm rat}\ ,\]
  such that
  \[  H^\ast(h_{tr}^2(S),\QQ)= H^2_{tr}(S,\QQ)\ ,\ \ H^\ast(h^2_{alg}(S),\QQ)=NS(S)_{\QQ}\ \]
  (here $H^2_{tr}(S)$ is defined as the orthogonal complement of the N\'eron--severi group $NS(S)_{\QQ}$ in $H^2(S,\QQ)$),
  and
   \[ A^\ast(h_{tr}^2(S))_{\QQ}=A^2_{AJ}(S)_{}\ .\]
   (The motive $h_{tr}^2(S)$ is called the {\em transcendental part of the motive\/}.)
   \end{theorem}

 \section{A result of Vial's} This section contains a ``Bloch conjecture'' type of statement. As already shown in \cite{Ch}, this statement is very useful in dealing with Voisin's conjecture on $0$-cycles.
 
 \begin{definition} Let $M\in\MM_{\rm rat}$ and let $X$ be a smooth projective variety. We say that $M$ is {\em motivated by $X$\/} if $M$ is isomorphic to a direct summand of a sum of tensor powers of motives of the form $h(X)(j)$, $j\in\ZZ$.
 \end{definition}

 \begin{theorem}[Vial \cite{Ch}]\label{ch} Let $M\in\MM_{\rm rat}$ be motivated by an abelian variety of dimension $\le 2$.  Assume that 
   \[ H^{i,j}(M)=0\ \ \ \hbox{for\ all\ }j<n\ .\]
   Then also
   \[ A_i(M)=0\ \ \ \hbox{for\ all\ }i<n\ .\]
    \end{theorem}
    
    \begin{proof} This is not stated verbatim in \cite{Ch}, but the argument is the same as that of \cite[Theorem 4.7]{Ch}. In a nutshell, the point is that (as proven in \cite[Corollary 3.13]{Ch})
    $M$ satisfies a strong form of the generalized Hodge conjecture, i.e. there is equality
      \[   N^r_H H^i(M) = \Gamma_\ast H^{i-2r}(A)\ ,\]
      where $A$ is a disjoint union of abelian varieties and $\Gamma$ is a correspondence from $A$ to $M$. (Here, $N^\ast_H$ denotes the {\em Hodge coniveau filtration\/} \cite[Definition 1.4]{Ch}.)
      
      Writing $M=(X,p,m)\in\MM_{\rm rat}$, the cohomological assumption thus translates into the fact that the cohomology class of $p$ factors as
      \[   h(X)\ \xrightarrow{\Psi}\ h(A)(n-m)\ \xrightarrow{\Xi}\ h(X)\ ,\]
      where $A$ is a disjoint union of abelian varieties, and $\Psi$ and $\Xi$ are correspondences in $A^\ast(X\times A)$ resp. in $A^\ast(A\times X)$. Since $M$ is Kimura finite-dimensional, one can apply the nilpotence theorem to $p-p\circ\Xi\circ\Psi\circ p$; the outcome is that the rational equivalence class of $p$ factors as
           \[   h(X)\ \xrightarrow{\Psi^\prime}\ h(A)(n-m)\ \xrightarrow{\Xi^\prime}\ h(X)\ .\]
           Taking Chow groups, this proves the theorem.
       \end{proof}

 \section{Main result}
 
 \begin{theorem}\label{main} Let $S$ be a Cancian--Frapporti surface. For any $a,b,c\in A^2_{AJ}(S)_{\ZZ}$, there is equality
   \[ a\times b\times c - b\times a\times c - c\times b\times a - a\times c\times b + b\times c\times a +c\times a\times b=0\ \ \ \hbox{in}\ A^6(S^3)_{\ZZ}\ .\]
   \end{theorem}
   
   \begin{proof} A first reduction step is that thanks to Roitman \cite{Ro}, one may replace $A^\ast()_{\ZZ}$ by Chow groups with $\QQ$-coefficients $A^\ast()$.
   
  Next, let us consider the decomposition of the Chow motive of $S$
   \[ h(S)=h^0(S)\oplus h^1(S)\oplus h^2_{tr}(S)\oplus h^2_{alg}(S)\oplus h^3(S)\oplus h^4(S)\ \ \ \hbox{in}\ \MM_{\rm rat}\ ,\]
   where $h^2_{tr}(S)$ is the {\em transcendental part\/} of the motive of $S$ (theorem \ref{tr}).
   
   The dominant morphism $\beta\colon S\to Y$ (proof of theorem \ref{pp}) identifies the motive of $Y$ with a submotive of the motive of $S$, in particular this gives (non-canonical) splittings
   \begin{equation}\label{spl} \begin{split}  h^2_{tr}(S)&=h^2_{tr}(Y)\oplus M_{tr}=h^2_{tr}(A)\oplus M_{tr}\ ,\\
                          h^2_{alg}(S)&=h^2_{alg}(Y)\oplus M_{alg}\ \ \ \ \ \ \hbox{in}\ \MM_{\rm hom}\ .\\   
                          \end{split}\end{equation}
     The surfaces $S$ and $Y$, being dominated by a product of curves, have finite-dimensional motive. This implies (using the nilpotence theorem \cite{Kim})
     that the splittings (\ref{spl}) also exist on the level of $\MM_{\rm rat}$.                   
                          
     We remark that the motive $M:=M_{tr}\oplus M_{alg}$ has
      \begin{equation}\label{hpq} \begin{split} \dim_\C H^{2,0}(M_{})&= \dim_\C H^{2,0}(S) - \dim_\C H^{2,0}(A)=2-1=1\ ,\\
                             \dim_\C H^{1,1}(M_{})&= \dim_\C H^{1,1}(S) - \dim_\C H^{1,1}(Y)=   7-5=2 .\\   
                             \end{split}\end{equation}
  
%
%
%

One has $A^\ast(h^2_{tr}(S))=A^2_{AJ}(S)$ and $A^\ast(h^2_{tr}(A))=A^2_{(2)}(A)=A^2_{AJ}(A)$ (here and below, for any abelian variety $A$, we write
$A^\ast_{(\ast)}(A)$ for the Fourier decomposition of \cite{Beau}, and $\pi^j_A$ for the Chow--K\"unneth projectors inducing the Fourier decomposition as in \cite{DM}).
This splitting of $h^2_{tr}(S)$ induces a splitting
  \[   A^2_{AJ}(S)  =A^2_{(2)}(A)\oplus  A^2(M_{tr})\ .\]
  
 We make two claims, that deal with the two pieces of this splitting separately:
 
 \begin{claim}\label{cl1} For any $a_1,a_2\in  A^2_{(2)}(A)$, there is equality
   \[ a_1\times a_2=a_2\times a_1\ \ \ \hbox{in}\ A^4(A\times A)\ .\]
   \end{claim}
   
 \begin{claim}\label{cl2} For any $v_1,v_2\in  A^2_{}(M_{tr})$, there is equality
   \[ v_1\times v_2=v_2\times v_1\ \ \ \hbox{in}\ A^4(S\times S)\ .\]
   \end{claim}

   Because of the equality
   \[ \bigwedge^3 \Bigl( A^2_{(2)}(A)\oplus  A^2(M_{tr})\Bigr) = \bigoplus_{j=0}^3 \bigwedge^j     A^2_{(2)}(A)\otimes \bigwedge^{3-j}  A^2(M_{tr})  \ ,\]
   these two claims together suffice to prove theorem \ref{main}.
   
   The first claim is easy, and directly follows from a more general result of Voisin's (this is \cite[Example 4.40]{Vo}):
        
   \begin{proposition}[Voisin \cite{Vo}] Let $A$ be an abelian variety of dimension $g$. Let $a_1,a_2\in A^g_{(g)}(A)$. Then
  \[a_1\times a_2=(-1)^g \, a_2\times a_1\ \ \ \hbox{in}\ A^{2g}(A\times A)\ .\]
  \end{proposition}
  
  In order to prove the second claim, we first need to understand the motive $M_{tr}$ a bit better.
  
  \begin{proposition}\label{M} There exist an abelian surface $B$, and a correspondence inducing a surjection
    \[  H^2(B\times B,\QQ)\ \twoheadrightarrow\ H^2(M_{tr},\QQ)\ .\]
    \end{proposition}
  
  \begin{proof} This follows from the specific geometry of the construction of $S$. Reverting to the notation of the proof of theorem \ref{pp}, the covering morphism $C_4\times C_4\to S$ induces 
 a surjection
   \[ H^2_{tr}(C_4\times C_4,\QQ)\ \twoheadrightarrow\ H^2_{tr}(S,\QQ)\ .\]
   An application of the K\"unneth formula gives a surjection
   \[ H^1(C_4,\QQ)\otimes H^1(C_4,\QQ)\ \twoheadrightarrow\ H^2_{tr}(C_4\times C_4,\QQ)\ .\]   
   The Abel--Jacobi map of the curve $C_4$ into the $4$-dimensional abelian variety $A_4:=\hbox{Jac}(C_4)$ induces  
 an isomorphism
   \[ H^1(A_4,\QQ)\otimes H^1(A_4,\QQ)\ \xrightarrow{\cong}\ H^1(C_4,\QQ)\otimes H^1(C_4,\QQ)\ .\]
  Choosing base points for the Abel--Jacobi maps in a compatible way, the triple covering of curves $C_4\to C_2$ induces a surjective homomorphism $A_4\to A:= \hbox{Jac}(C_2)$. Using Poincar\'e's complete reducibility theorem, this implies that $A_4$ is isogenous to $B\times A$, where $B$ is an abelian surface. This gives a decomposition
     \[ H^1(A_4,\QQ)= H^1(A\times B,\QQ)=H^1(A,\QQ)\oplus H^1(B,\QQ)\ .\]
     Combining all these maps, we obtain a surjection
     \begin{equation}\label{incl} 
    \begin{split} \bigl( H^1(A,\QQ)\oplus H^1(B,\QQ)\bigr)^{\otimes 2}\ \xrightarrow{\cong}\ H^1(C_4,\QQ)\otimes H^1(C_4,\QQ) \twoheadrightarrow\ H^2_{tr}(C_4\times C_4,\QQ)&\\\twoheadrightarrow H^2_{tr}(S,\QQ)
    \xrightarrow{\cong}
       H^2_{tr}(A,\QQ)\oplus H^2&(M_{tr},\QQ)\ . \\
    \end{split}  
          \end{equation}
  It follows from the truth of the standard conjectures for surfaces and abelian varieties that all arrows in (\ref{incl}) are induced by correspondences. Let us now consider the summand
  $H^1(A,\QQ)\otimes H^1(A,\QQ)$ of the left-hand side of (\ref{incl}).
  The triple covering $C_4\to C_2$ induces a commutative diagram
   \[ \begin{array}[c]{ccccc}
      H^1(C_4,\QQ)\otimes H^1(C_4,\QQ) &\to& H^2_{tr}(\hbox{Sym}^2 C_4^{},\QQ)&\to& H^2_{tr}(S,\QQ)\\
      \downarrow&&\downarrow&&\downarrow{\scriptstyle \alpha_\ast}\\
        H^1(C_2,\QQ)\otimes H^1(C_2,\QQ) &\to& H^2_{tr}(\hbox{Sym}^2 C_2^{},\QQ)&\xrightarrow{\pi_\ast}& H^2_{tr}(A,\QQ)\ ,\\   
        \end{array}\]
        where the composition of upper horizontal arrows is the same map $H^1(C_4,\QQ)\otimes H^1(C_4,\QQ)\to H^2_{tr}(S,\QQ)$ as in (\ref{incl}), and $\alpha$ and $\pi$ are as in the proof of theorem \ref{pp}. Because the summand $H^1(A,\QQ)\otimes H^1(A,\QQ)$ of
        $H^1(C_4,\QQ)\otimes H^1(C_4,\QQ)$ maps isomorphically to $H^1(C_2,\QQ)\otimes H^1(C_2,\QQ)$, it follows that this summand maps onto $H^2_{tr}(A,\QQ)$ in (\ref{incl}).
        More precisely, the map 
          \[ H^1(A,\QQ)\otimes H^1(A,\QQ)\ \to\  H^2_{tr}(A,\QQ)\oplus H^2(M_{tr},\QQ)\] 
     deduced from diagram (\ref{incl}) induces a surjection onto $H^2_{tr}(A,\QQ)$ and the zero-map to $H^2(M_{tr},\QQ)$, under both projections.
                  

Let us now analyze the other summands of the left-hand side of (\ref{incl}).
 There is an induced action of $\xi\in\aut(C_4)$ on $A_4$, and an eigenspace decomposition
    \[ H^1(A_4,\C)= H^1(A_4,\C)^{(1)}\oplus H^1(A_4,\C)^{(\nu)}\oplus H^1(A_4,\C)^{(\nu^2)}\ \]
    (where $\nu$ is a primitive third root of unity). The first eigenspace (which is $2$-dimensional) corresponds to $H^1(A,\C)\cong H^1(C_2,\C)$, while the sum of the two other ($1$-dimensional) summands corresponds to $H^1(B,\C)$. 
    The covering morphism $C_4\times C_4\to S$ factors as
     \[ C_4\times C_4\ \to\ (C_4\times C_4)/\langle \xi_{xy}\rangle\ \to\ S\ \]
     (where $\xi_{xy}\in\aut(C_4\times C_4)$ is the order $3$ automorphism acting diagonally as in the proof of theorem \ref{pp}),
     and so there is a factorization
     \[    H^2_{}(C_4\times C_4,\C)\ \to\ H^2_{}((C_4\times C_4)/\langle \xi_{xy}\rangle,\C)\ \to\ H^2_{}(S,\C) \ .\]
     It follows that the summands of type $H^1(A_4,\C)^{(1)}\otimes H^1(A_4,\C)^{(\nu)}$ and $H^1(A_4,\C)^{(1)}\otimes H^1(A_4,\C)^{(\nu^2)}$ (and their permutations) map to zero under the natural map. In other words, the natural map 
     \[ H^1(A_4,\C)\otimes H^1(A_4,\C)\ \to\ H^2(S,\C) \]
     is the same as the composition
     \[   \begin{split}  &H^1(A_4,\C)^{(1)}\otimes H^1(A_4,\C)^{(1)}\\
     &\oplus \Bigl(  H^1(A_4,\C)^{(\nu)}\otimes H^1(A_4,\C)^{(\nu^2)}\oplus  H^1(A_4,\C)^{(\nu^2)}\otimes H^1(A_4,\C)^{(\nu)}  \Bigr)      \ \to\ H^2(S,\C)\ . \end{split}\]     
     The first summand corresponds to $H^1(A,\C)\otimes H^1(A,\C)$, the second is contained in $H^1(B,\C)\otimes H^1(B,\C)$. Thus, we see that ``mixed terms'' $H^1(A,\QQ)\otimes H^1(B,\QQ)$ 
     and $H^1(B,\QQ)\otimes H^1(A,\QQ)$ in (\ref{incl}) map to zero. It follows that the summand $H^1(B,\QQ)\otimes H^1(B,\QQ)$ in (\ref{incl}) maps onto $H^2(M_{tr},\QQ)$.
        \end{proof}

 Let us now prove claim \ref{cl2} (and hence theorem \ref{main}). Proposition \ref{M}, in combination with the fact that the standard conjectures hold for surfaces and abelian varieties, shows that there is a map
  \[ M_{tr}\ \to\ h^2(B\times B)\ \ \ \hbox{in}\ \MM_{\rm hom} \]
  admitting a left-inverse. Using Kimura finite-dimensionality (cf. for instance \cite[Section 3.3]{V3}), the same holds in $\MM_{\rm rat}$, i.e. the motive $M_{tr}$ is {\em motivated by the abelian surface $B$\/}. The motive $M:=\wedge^2 M_{tr}$ (being a submotive of $M_{tr}^{\otimes 2}$) is also motivated by $B$. The motive $M$ has $H^j(M)=0$ for all $j\not=4$ and $H^{4,0}(M)=\wedge^2 H^{2,0}(M_{tr})=0$, since $\dim H^{2,0}(M_{tr})=1$ (cf. (\ref{hpq})). Applying theorem \ref{ch} to $M$ (with $n=1$), we find that
   \[  \wedge^2 A_0(M_{tr})=  A_0(M)=0\ , \]
   proving claim \ref{cl2}.
  \end{proof}

  \begin{corollary}\label{ghc} Let $S$ be a Cancian--Frapporti surface, and let $m>2$. Then the sub-Hodge structure
   \[ \wedge^m H^2(S,\QQ)\ \subset\ H^{2m}(S^m,\QQ) \]
   is supported on a divisor.
   \end{corollary}
   
   \begin{proof} As Voisin had already remarked \cite[Corollary 3.5.1]{V9}, this is implied by the truth of conjecture \ref{conj} for $S$ (as can be seen using the Bloch--Srinivas argument \cite{BS}).
   \end{proof}

%

  \begin{remark}\label{notsure} The strong form of the generalized Hodge conjecture (as mentioned in the proof of theorem \ref{ch}) is a result specific to self-products of abelian {\em surfaces\/}, and seems out of reach for self-products of higher-dimensional abelian varieties. As such, the argument employed here crucially hinges on the fact that the Cancian--Frapporti surfaces $S$ are constructed starting from a Galois cover $C_m\to C_n$, where $C_m, C_n$ are curves of genus $m$ resp. $n$ and $m-n\le 2$.
 While the other surfaces with $p_g=q=2$ constructed in \cite{CF} still have surjective Albanese map \cite[Theorem 4]{P}, for all but one of them the difference $m-n$ is larger than $2$. As such, they do not enter in the set-up of the present note; some new argument is needed to prove conjecture \ref{conj} for them.
   \end{remark}
  
  \begin{remark} My initial hope was to establish that Cancian--Frapporti surfaces have a {\em multiplicative Chow--K\"unneth decomposition\/} (in the sense of \cite{SV}), and satisfy the condition $(\ast)$ of \cite{FV}. This proved to be unfeasibly difficult, however. 
  
(The problem was that I could not prove that the class of the curve $C_4$ in $A^3(A_4)$ is symmetrically distinguished. This cannot possibly be true for a {\em general\/} genus $4$ curve, but might perhaps be true for $C_4$ because it is a triple cover over $C_2$ ?)
  \end{remark}

\vskip1cm
\begin{nonumberingt} I am grateful to a referee who kindly suggested substantial simplifications of the main argument. Thanks to
Kai and Len, my dedicated coworkers at the Alsace Center for Advanced Lego-Building and Mathematics.
\end{nonumberingt}

\end{document}